\documentclass[12pt]{article}
\usepackage{amsmath,amssymb,amsfonts,amsthm}
\newcounter{item}[section]
\newcounter{kirshr}
\newcounter{kirsha}
\newcounter{kirshb}
\newenvironment{enumroman}{\setcounter{kirshr}{1}
\begin{list}{(\roman{kirshr})}{\usecounter{kirshr}} }{\end{list}}
\newenvironment{enumarab}{\setcounter{kirshb}{1}
\begin{list}{(\arabic{kirshb})}{\usecounter{kirshb}} }{\end{list}}
\newtheorem{theorem}{Theorem}[section]

\theoremstyle{definition}

\newtheorem{definition}[theorem]{Definition}
\def\R{\mathbb{R}}

\def\C{{\mathfrak{C}}}

\def\Nr{{\mathfrak{Nr}}}

\def\Sg{{\mathfrak{Sg}}}

\def\RCA{{\sf RCA}}

\def\A{{\mathfrak{A}}}
\def\B{{\mathfrak{B}}}
\def\C{{\mathfrak{C}}}

\def\M{{\mathfrak{M}}}

\def\Bl{{\mathfrak{Bl}}}
\def\CA{{\sf CA}}

\def\SC{{\bf SC}}
\def\QEA{{\bf QEA}}

\def\K{{\bf K}}
\def\K{{\bf K}}

\def\RCA{{\sf RCA}}

\def\Rd{{\ Rd}}
\def\(R)RA{{\bf (R)RA}}

\def\R{\mathbb{R}}
\def\F{{\sf F}}

\def\c #1{{\cal #1}}
 \def\CA{{\sf CA}}
\def\B{{\sf B}}

\def\K{{\sf K}}

 \def\Cm{{\mathfrak{Cm}}}
\def\Nr{{\mathfrak{Nr}}}

\def\Nr{{\mathfrak{Nr}}}
\def\Tm{{\mathfrak{Tm}}}
\def\A{{\mathfrak{A}}}
\def\B{{\mathfrak{B}}}
\def\C{{\mathfrak{C}}}

\def\E{{\mathfrak{E}}}
\def\A{{\mathfrak{A}}}
\def\B{{\mathfrak{B}}}
\def\C{{\mathfrak{C}}}

\def\E{{\mathfrak{E}}}

\def\P{{\mathfrak{P}}}

\def\Bb{{\mathfrak{Bb}}}
\def\L{{\mathfrak{L}}}

\def\Bb{{\mathfrak{Bb}}}
\def\L{{\mathfrak{L}}}
\def\CA{{\sf CA}}

\def\RCA{{\sf RCA}}

\def\At{{\sf At}}

\def\R{{\sf R}}
\def\Rd{{\sf Rd}}

\def\pa{$\forall$}

\def\ws{winning strategy}

\def\QEA{{\sf QEA}}
\def\SC{{\sf SC}}

\def\Cof{{\sf Cof}}
\def\pe{$\exists$}
\def\Cof{\sf Cof}
\title{On neat atom structures for cylindric-like algebras}
\author{Tarek Sayed Ahmed }
\begin{document}
\maketitle

\begin{abstract} 
\begin{definition}
\begin{enumarab} 
\item Let $1\leq k\leq \omega$. Call an atom structure $\alpha$ weakly $k$ neat representable, 
the term algebra is in $\RCA_n\cap \Nr_n\CA_{n+k}$, but the complex algebra is not representable.
\item Call an atom structure neat if there is an atomic algebra $\A$, such that $\At\A=\alpha$, $\A\in \Nr_n\CA_{\omega}$ and for every algebra
$\B$ based on this atom structure there exists $k\in \omega$, $k\geq 1$, such that $\B\in \Nr_n\CA_{n+k}$.
\item Let $k\leq \omega$. Call an atom structure $\alpha$ $k$ complete, 
if there exists $\A$ such that $\At\A=\alpha$ and $\A\in S_c\Nr_n\CA_{n+k}$.
\item Let $k\leq \omega$. Call an atom structure $\alpha$ $k$ neat if there exists $\A$ such that $\At\A=\alpha$, 
and $\A\in \Nr_n\CA_{n+k}.$ 
\end{enumarab}
\end{definition}
\begin{definition} Let $\K\subseteq \CA_n$, and $\L$ be an extension of first order logic. We say that
$\K$ is well behaved w.r.t to $\L$, if for any $\A\in \K$, $\A$ atomic, and for any any atom structure $\beta$ such that $\At\A\cong \beta$,
for any $\B$, $\At\B=\beta$, $\B\in K$
\end{definition}
We investigate the existence of such structures, and the interconnections.
We also present several $\K$s and $\L$s as in the second definition. All our results extend to Pinter;s algebras and polyadic algebras with and without equality.
\end{abstract}

We prove:
\begin{theorem} 
\begin{enumarab}
\item There exists a countable weakly $k$ neat atom structure if and only if $k<\omega$
\item There exists an atom structure of a representable algebra that is not neat, this works for all dimensions.
\item  There exists an atom structure that is $n+2$ complete, that is elementary equivalent an an atom structure
that is m neat for all $m\in \omega$.
\item The class of completely representable algebras is not well behaved with repect to $L_{\omega,\omega}$ while the clas
of neat reducts is not well behaved with respect to $L_{\infty,\omega}$
\end{enumarab}
\end{theorem}

\begin{proof} 
\begin{enumarab}
\item We prove 3 and the first part of 4 together. 
We will not give the details, because the constructionwe use a rainbow construction for cylindric algebras,
is really involved nd it wil be submitted elsewhere,
however, we give the general idea.
We use essentially the techniques in \cite{r}, together with those in \cite{hh}, extending the rainbow construction
to cylindric algebra. But we mention a very important difference. 

In \cite{hh} one game is used to test complete representability.
In  \cite{r} {\it three} games were divised testing different neat embedability properties.
(An equivalence between complete representability and special neat embeddings is proved in \cite{IGPL})

Here we use only two games adapted to the $\CA$ case. This suffices for our purposes. 
The main result in \cite{hh}, namely, that the class of completely representable algebras of dimension
$n\geq 3$, is non elementary, follows from the fact that \pe\  cannot win the infinite length 
game, but he can win the finite ones. 

Indeed a very useful way of characterizing non elementary classes, 
say $\K$, is a Koning lemma argument. The idea is  to devise a game $G$ on atom structures such that for a given algebra atomic $\A$  \pe\ 
has a winning strategy on its atom structure for all games of finite length, 
but \pa\ wins the $\omega$ round game. It will follow that there a countable cylindric algebra $\A'$ such that $\A'\equiv\c
A$ and \pe\ has a \ws\ in $G(\A')$.
So $\A'\in K$.  But $\A\not\in K$
and $\A\preceq\A'$. Thus $K$ is not elementary.

To obtain our results we use {\it two} distinct games, both having $\omega$ rounds, played on a rainbow atom structure, the desired algebra is any algebra based on this atom structute it can be the term algebar generated by the atoms or the 
full complex algebra.  
Of course the games are very much related.

In this new context \pe\ can also win a finite game with $k$ rounds for every $k$. Here the game
used  is more complicated than that used in Hirsch and Hodkinson \cite{hh}, 
because in the former case we have three kinds of moves which makes it harder for \pe\ 
to win. 

Another difference is that the second game, call it $H$,  is actually  played on pairs, 
the first component is an atomic network (or coloured graph)  defined in the new context of cylindric 
algebras, the second is a set of hyperlabels, the finite sequences of nodes are labelled, 
some special ones are called short, and {\it neat} hypernetworks or hypergraphs are those that label short hyperedges with the same label. 
And indeed a \ws\ for \pe\ 
in the infinite games played on an atom structure forces that this is the atom structure of a neat reduct; 
in fact an algebra in $\Nr_n\CA_{\omega}$. However, unlike complete representability,
does not exclude the fact, in principal, there are other representable algebras 
having the same atom stucture can be only subneat reducts.

But \pe\ cannot win the infinite length game, it can only win the finite length games of length $k$ for every finite $k$.

On the other hand, \pa\  can win 
{\it another  pebble game}, also in $\omega$ rounds (like in \cite{hh} on a red clique), but
 there is a  finiteness condition involved in the latter, namely  is the number of nodes 'pebbles 'used, which is $k\geq n+2$, 
and  \pa\ s \ws\ excludes the neat embeddablity of the algebra in $k$ extra dimensions. This game will be denoted by $F^k$. 

This implies that $\A$ is elementary equivalent to a full neat reduct but it is not in 
$S_c\Nr_n\CA_{n+2}$.

And in fact the Hirsch Hodkinson's main result in \cite{r}, 
can be seen as a special case, of our construction. The game $F^k$, without the restriction on number
of pebbles used and possibly reused, namely $k$ (they have to be reused when $k$ is finite), but relaxing the condition of finitness,
\pa\ does not have to resuse node, and then this game  is identical to the game $H$ when we delete the  hyperlabels from the latter, 
and forget about the second and third 
kinds of move. So to test only complete representability, we use only these latter games, which become one, namely the one used 
by Hirsch and Hodkinson in \cite{hh}. 

In particular, our algebra $\A$ constructed is not completely representable, but is elementary equivalent to one that is.
This also implies that the class of completely representable atom structures are not elementary, the atom structure of the 
former two structures are elementary equivalent, one is completely representable, the other is not.
Since an atom structure of an algebra is first order interpretable in the algebra, hence, 
the latter also gives an example of an atom structure that is 
weakly representable but not strongly representable.

\item Now we prove 2, Let $k$ be a cardinal. Let $\E_k=\E_k(2,3)$ denote the relation algebra
which has $k$ non-identity atoms, in which $a_i\leq a_j;a_l$ if $|\{i,j,l\}|\in \{2,3\}$
for all non-identity atoms $a_i, a_j, a_k$.(This means that all triangles are allowed except the monochromatic ones.) 
These algebras were defined by Maddux.
Let $k$ be finite, let $I$ be the set of non-identity atoms of $\E_k(2,3)$ and let $P_0, P_1\ldots P_{k-1}$ be an enumeration of the elements of $I$.
Let $l\in \omega$, $l\geq 2$ and let $J_l$ denote the set of all subsets 
of $I$ of cardinality $l$. Define the symmetric ternary relation on $\omega$ by $E(i,j,k)$ if and only if $i,j,k$ are evenly distributed, that is
$$(\exists p,q,r)\{p,q,r\}=\{i,j,k\}, r-q=q-p.$$
Now assume that $n>2$, $l\geq 2n-1$, $k\geq (2n-1)l$, $k\in \omega$. Let $\M=\E_k(2,3).$
Then $\M$ is a simple, symmetric finite atomic relation algebra.
The idea is to blow up and blur $\M$. This is done by splitting each atom into infinitely countable many ones 
and using a finite set of blurs. So the underlying set of the new atom structure will be of the form
$\omega\times \At\M\times J$, $J$ is a set of finite blurs that corresponds to colours that in turn correpond to non principal 
ultarfilters, needed to represent the term algebra. This term algebra  which is blurred in the sense that $\M$ is not embeddable in it; 
but $\M$ will be embeddable in the full complex algebra the former can be only represented 
on finite sets, the later on infinite sets, if at all, hence
it cannot be represpentable. The idea used here is to define two partitions of the set $I\times \At\M\times J$, 
the first is used to embed $\M$ into the complex algebra, and the term algebra will be the second partition up to finite and 
cofinite deviations.

Now we have $$(\forall V_2\ldots, V_n, W_2\ldots W_n\in J_l)(\exists T\in J_l)(\forall 2\leq i\leq n)$$
$$(\forall a\in V_i)\forall b\in W_i)(\forall c\in T_i)(a\leq b;c).$$
That is $(J4)_n$ formulated in \cite{ANT} p. 72 is satisfied. Therefore, as proved in \cite{ANT} p. 77,
$B_n$ the set of all $n$ by $n$ basic matrices is a cylindric basis of dimension $n$.
But we also have $$(\forall P_2,\ldots ,P_n,Q_2\ldots Q_n\in I)(\forall W\in J_l)(W\cap P_2;Q_2\cap\ldots \cap P_n:Q_n\neq 0)$$
That is $(J5)_n$ formulated on p. 79 of \cite{ANT} holds. According to definition 3.1 (ii) $(J,E)$ is an $n$ blur for $\M$, 
and clearly $E$ is definable in $(\omega,<)$.
Let $\C$ be as defined in lemma 4.3 in \cite{ANT}.   
Then, by lemma 4.3, $\C$ is a subalgebra of $\Cm\B_n$, hence it contains the term algebra $\Tm\B_n$.
Denote $\C$ by $\Bb_n(\M, J, E)$. Then by theorem 4.6 in \cite{ANT} $\C$ is representable, and by theorem 4.4 in \cite{ANT} 
for $m<n$
$\Bb_m(\M,J,E)=\Nr_m\Bb_n(\M,J,E)$. However $\Cm\B_n$ is not representable.
In \cite{ANT} $\R=\Bb(\M,J,E)$ is proved to be generated by a single element. 
If $k=\omega$, then algebra in $\Nr_n\CA_{\omega}$ will be completely representable. If the term algebra is completely reprsentable, then the complex algebra will be 
representable.

\item Concerning that the class of strongly representable algebras, one uses an ultraproduct of what we call anti-Monk algebras.
If one increases the number of blurs in the above construction, then one gets a 
a sequence of non representable algebras, namely the complex algebras
based on the atom structure as defined above, with an increasing number of blurs. 
This corresponds to algebras based on graphs with increasing finite chromatic number ; 
the limit will be an algebra based on a graph of infinite chromatic number, hence
will be representable, in fact, completely representable. This for example proves Monk's classical non finite axiomatizability result.
A graph which has a finite colouring is called a bad graph by Hirsch and Hodkinson. A good graph is one which gives representable algebras,
it has infinite chromatic number. So the Monk theme is to construct algebra based on bad graphs that converge to one that is based
on  a good graph.
This theme is reversed  by used by what we call anti Monk algebras, that are based on Erdos graphs.
Every  graph in this sequence has infinite chromatic number and the limit  algebra based on the ultraproduct of these
graphs will be only  two colurable. This shows that the class of strongly atom structures is not elementary.

\item  Here the eaxmple works for all dimensions, infinite included.
Let $\alpha>1$ and $\F$ is field of characteristic $0$. 
Let 
$$V=\{s\in {}^{\alpha}\F: |\{i\in \alpha: s_i\neq 0\}|<\omega\},$$
Note that $V$is a vector space over the field $\F$. 
We will show that $V$ is a weakly neat atom structure that is not strongly neat.
Indeed $V$ is a concrete atom structure $\{s\}\equiv _i \{t\}$ if 
$s(j)=t(j)$ for all $j\neq i$, and
$\{s\}\equiv_{ij} \{t\}$ if $s\circ [i,j]=t$.

Let $\C$ be the full complex algebra of this atom structure, that is
$${\C}=(\wp(V),
\cup,\cap,\sim, \emptyset , V, {\sf c}_i,{\sf d}_{ij}, {\sf s}_{ij})_{i,j\in \alpha}.$$  
Then clearly $\wp(V)\in \Nr_{\alpha}\CA_{\alpha+\omega}$.
Indeed Let $W={}^{\alpha+\omega}\F^{(0)}$. Then
$\psi: \wp(V)\to \Nr_{\alpha}\wp(W)$ defined via
$$X\mapsto \{s\in W: s\upharpoonright \alpha\in X\}$$
is an isomomorphism from $\wp(V)$ to $\Nr_{\alpha}\wp(W)$.
We shall construct an algebra $\A$ such that $\At\A\cong V$ but $\A\notin \Nr_{\alpha}\CA_{\alpha+1}$.

Let $y$ denote the following $\alpha$-ary relation:
$$y=\{s\in V: s_0+1=\sum_{i>0} s_i\}.$$
Note that the sum on the right hand side is a finite one, since only 
finitely many of the $s_i$'s involved 
are non-zero. 
For each $s\in y$, we let 
$y_s$ be the singleton containing $s$, i.e. $y_s=\{s\}.$ 
Define 
${\A}\in \QEA_{\alpha}$ 
as follows:
$${\A}=\Sg^{\C}\{y,y_s:s\in y\}.$$
Then, 
$$\Rd_{ca}\A\notin \Nr_{\alpha}\SC_{\alpha+1}.$$ 
That is for no $\mathfrak{P}\in \SC_{\alpha+1}$, it is the case that $\Sg^{\C}X$ exhausts the set of all $\alpha$ dimensional elements 
of $\mathfrak{P}$. 

\item  $\R$ be an uncountable set and let $\Cof R$ be set of all non-empty finite or cofinite subsets  $R$.
Let $\alpha$ be any ordinal. For $k$ finite, $k\geq 1$, let
$$S(\alpha,k)=\{i\in {}^\alpha(\alpha+k)^{(Id)}: \alpha+k-1\in Rgi\},$$
$$\eta(X)=\bigvee \{C_r: r\in X\},$$
$$\eta(R\sim X)=\bigwedge\{\neg C_r: r\in X\}.$$


We give a construction for cylindric algebras for all dimensions $>1$.
Let $\alpha>1$ be any ordinal. $(W_i: i\in \alpha)$ be a disjoint family of sets each of cardinality $|\R|$.
Let $M$ be their disjoint union, that is
$M=\bigcup W_i$. Let $\sim$ be an equvalence relation on $M$ such that $a\sim b$ iff $a,b$ are in the same block.
Let $T=\prod W_i$. Let $s\in T$, and let $V={}^{\alpha}M^{(s)}$. 
For $s\in V$, we write $D(s)$ if $s_i\in W_i$, and we let $\C=\wp(V)$.
There are $\alpha$-ary relations $C_r\subseteq {}^{\alpha}M^{(s)}$ on the base $M$ for all $r\in \R$,
such that conditions (i)-(v) below hold:
\begin{enumroman}
\item $\forall s(s\in C_r\implies D(s))$

\item For all $f\in {}^{\alpha}W^{(s)}$ for all $r\in \R$, for all permutations
$\pi\in ^{\alpha}\alpha^{(Id)}$, if $f\in C_r$ then $f\circ \pi\in C_r.$ 

\item For all $1\leq k<\omega$, for all 
$v\in {}^{\alpha+k-1}W^{(s)}$ one to one,  for all $x\in W$, $x\in W_m$ say, then for any
function $g:S(\alpha,k)\to \Cof\R$ 
for which $\{i\in S(\alpha,k):|\{g(i)\neq \R\}|<\omega\}$, 
there is a $v_{\alpha+k-1}\in W_m\smallsetminus Rgv$ such that 
and  
$$\bigwedge \{D(v_{i_j})_{j<\alpha}\implies \eta(g(i))[\langle v_{i_j}\rangle]: 
i\in S(\alpha,k)\}.$$
\item The $C_r$'s are pairwise disjoint.
\end{enumroman}

For $u\in S_{\alpha}$ and $r\in R$, let 
$$p(u,r)=C_r\cap (W_{u_0}\times W_{u_1}\times W_{u_i}\times)\cap {}^{\alpha}W^{(s)}.$$ 
Let $$\A=\Sg^{\C}\{p(u,r): u\in S_{\alpha}: r\in \R\}.$$
For $u\in {}^{\alpha}\alpha^{(Id)}$, let $1_u=W_{u_0}\times W_{u_i}\times \cap V$
and $\A_u$ denote the relativisation of $\A$ to $1_u$.
i.e $$\A_u=\{x\in A: x\leq 1_u\}.$$ 
$\A_u$ is a boolean algebra. Also  $\A_u$ is uncountable and {\it atomic} for every $u\in V$
elements of $\A_u$.  
Because of the saturation condition above, we have $\A\in \Nr_{\alpha}\CA_{\alpha+\omega}$.
Define as above  map $f: \Bl\A\to \prod_{u\in {}V}\A_u$, by
$$f(a)=\langle a\cdot \chi_u\rangle_{u\in{}V}.$$
Let $\L$ be the quantifier free reduct of $L_{\infty, \omega}$ with infinite conjunctions, and quantifier free reduct of $L_{\omega,\omega}$
for finite dimensions.
We will expand the language of the boolean algebra $\prod_{u\in V}\A_u$ by constants in 
such a way that
$\A$ becomes $\L$ interpretable in the expanded structure.
We shall give more details here, because the meta-logic is infinitary.
As before $\P$ denote the 
following structure for the signature of boolean algebras expanded
by constant symbols $1_u$ for $u\in {}V$ and ${\sf d}_{ij}$ for $i,j\in \alpha$: 
\begin{itemize}
\item The Boolean part of $\P$ is the boolean algebra $\prod_{u\in {}V}\A_u$,

\item $1_u^{\P}=f(\chi_u^{\M})=\langle 0,\cdots0,1,0,\cdots\rangle$ 
(with the $1$ in the $u^{th}$ place)
for each $u\in {}V$,

\item ${\sf d}_{ij}^{\P}=f({\sf d}_{ij}^{\A})$ for $i,j<\alpha$.
\end{itemize}

Define a map $f: \Bl\A\to \prod_{u\in {}V}\A_u$, by
$$f(a)=\langle a\cdot \chi_u\rangle_{u\in{}V}.$$

We now show that $\A$ is $\L$ interpretable in $\P.$  
For this it is enough to show that 
$f$ is one to one and that $Rng(f)$ 
(Range of $f$) and the $f$-images of the graphs of the cylindric algebra functions in $\A$ 
are definable in $\P$. Since the $\chi_u^{\M}$ partition 
the unit of $\A$,  each $a\in A$ has a unique expression in the form
$\sum_{u\in {}V}(a\cdot \chi_u^{\M}),$ and it follows that 
$f$ is boolean isomorphism: $bool(\A)\to \prod_{u\in {}V}\A_u.$
So the $f$-images of the graphs of the boolean functions on
$\A$ are trivially definable. 
$f$ is bijective so $Rng(f)$ is 
definable, by $x=x$. For the diagonals, $f({\sf d}_{ij}^{\A})$ is definable by $x={\sf d}_{ij}$.

Finally we consider cylindrifications for $i<\alpha$. Let $S\subseteq {}V$ and  $i,j<\alpha$, 
let $t_S$ and $h_S$ be the closed infinitary terms:
$$\sum\{1_v: v\in {}V, v\equiv_i u\text { for some } u\in S\}.$$
 
Let
$$\eta_i(x,y)=\bigwedge_{S\subseteq {}V}(\bigwedge_{u\in S} x.1_u\neq 0\land 
\bigwedge_{u\in {}V\smallsetminus S}x.1_u=0\longrightarrow y=t_S).$$
These are well defined.
Then it can be proved that for all $a\in A$, $b\in P$, we have 
$$\P\models \eta_i(f(a),b)\text { iff } b=f({\sf c}_i^{\A}a).$$
(For finite dimensions all this can be implemented in the quantifier free reduct of first order logic, we do not need infinite conjunctions).

Now we can deduce that there is an algebra $\P$ that is $\L$ interpretable in $\A$, with a complete elementary subalgebra that is not a neat reduct.
For finite dimensions, $\L$ is just the quantifier free reduct of first order logic, and we obtain our previous results for cylindric algebras, 
with the additional condition that our algebras 
are atomic. 

If we take $\B$ to be $\Sg^{\A}\{p(u,r): r\in \R, u\neq Id\cup p(u,r): r\in N\},$ then $\B$ is an elementary complete subalgebra of $\A$.
This works for all dimensions, and basically follows from that fact that $\A$ has a very rich group of automorphisms, 
every permutation of $P(Id)=\{p(Id,r): r\in \R\}$ induces one that is the identity 
on $P\sim P(Id)$.

\item Here we play a game between $\At\A$ and $\At\B$ that will show that they are $L_{\infty,\omega}$ equivalent, then so are $\A$ and $\B$,
because the atom structure of an algebra is interpretable in the algebra, we can play the games on the atoms
of the algebra.
For this purpose, we devise a game between $\forall$ and $\exists$.
The game is played in $\mu\leq \omega$
steps. At the $i$th step of a play, player $\forall$ takes one of the structures $\A$, $\B$ 
and 
chooses an atom of this structure; then $\exists$ chooses 
an atom of the other structure. So between them they choose an atom $a_i$ of $\A$ 
and an atom 
$b_i$ of $\B$. 
the play  sequences $\bar{a}=(a_i:i<\mu)$ and $\bar{b}=(b_i: i<\mu)$ have been chosen.
The pair $(\bar{a}, \bar{b})$ is known as the play.
We count the play $(\bar{a}, \bar{b})$ as a win for player $\exists$, 
and we say that $\exists$ wins the play, if there is an isomorphism
$f:\Sg^{\A}ran({\bar{a}})\to \Sg^{\B}ran({\bar {b}})$ such that $f\bar{a}=\bar{b}.$
Let us denote this game by $EF_{\mu}(\A,\B).$ (It is an instance of an 
Ehrenfeuch-Fraisse pebble game)
Two atomic structures $\A$ and $\B$ are back and forth equivalent 
if $\exists$ has a winning strategy 
for the game $EFA_{\omega}(\A,\B)$.  
For $u\in S_n$, let $$1_u=W_{u_0}\times W_{u_1}\ldots \times W_{u_{n-1}}.$$
then $\{1_u: u\in S_n\}$ forms a partition of the unit
$^nW$ of $\A(n)$. 
It is easy to see that $1_u\in \A(n)\cap \B(n)$. Let $\A_u=\{x\in \A(n): x\leq 1_u\}$
and $\B_u=\{x\in \B(n): x\leq 1_u\}$. Then $\A_u$ and $\B_u$ are atomic boolean algebras.
The set of atom of $A_u$ is 
$P(u)=\{p(u,r):r\in R\}$ while that of $\B_u$ is $P(u)$ if $u\notin T_n$
and $P_{\omega}(u)$ otherwise.  
For all nonzero $a\in Rl_{1_u}A(n)$, for all $i<n$, ${\sf c}_ia={\sf c}_i1_u$.
For all $a\in A$, for all $i<n$, ${\sf c}_ia\cap 1_{u}\in \{0,1_{u}\}$.
Hence both $\A$ and $\B$ are atomic . So $\A$ and $\B$
are identical in all components except for the components ''coloured " 
by $1_{u}$, $u\in T_n$ beneath which $\A$ has 
uncountably many atoms and $\B$ has countably many atoms.
Now for the game. 
At each step, if the play so far $(\bar{a}, \bar{b})$ and $\forall$ chooses an atom $a$ 
in one of the substructures, we have one of two case. 
Either $a.1_u=a$ for some $u\notin T_n$
in which case
$\exists$ chooses the same atom in the other structure. 
Else $a\leq 1_{u}$ for some $u\in T_n.$ 
Then
$\exists$ chooses a new atom below $1_{u}$ 
(distinct from $a$ and all atoms played so far.)
This is possible since there finitely many atoms in 
play and there are infinitely many atoms below
$1_{u}$.
This strategy makes $\exists$ win. 
Let $J$ be the corresponding  back and forth system.   
Order $J$ by reverse inclusion, that is $f\leq g$ 
if $f$ extends $g$. $\leq$ is a partial order on $J$.
For $g\in J$, let $[g]=\{f\in J: f\leq g\}$. Then $\{[g]: g\in J\}$ is the base of a 
topology on 
$J.$ Let $\C$ be the complete  
Boolean algebra of regular open subsets of $J$ with respect to the topology 
defined on $J.$
Form the boolean extension $\M^{\C}.$
Then $G$ is an isomorphism in $\M^{\C}$ of $\breve{\A}$ to 
$\breve{\B}.$
We shall use the following for $s\in \M^{\C}$, (1):
$$||(\exists x\in \breve{s})\phi(x)||=\sum_{a\in s}||\phi(\breve{a})||.$$
Define $G$ by
$||G(\breve{a},\breve{b})||=\{f\in {J}: f(a)=b\}$. 
for $c\in A$ and $d\in B$. It can be checked that $G$ is a well defined isomorphism.
Since $\A$ and $\B$ are isomorphic in a Boolean extension of the universe of sets, then they are $L_{\infty, \omega}$ equivalent.

Another way is to form a boolean extension $\M^*$ of $\M$
in which the cardinalities of $\A$ and $\B$ collapse to 
$\omega$.  Then $\A$ and $\B$ are still back and forth equivalent in $\M^*.$
Then $\A\equiv_{\infty\omega}\B$ in $\M^*$, and hence also in $\M$
by absoluteness of $\models$.

\end{enumarab}

\end{proof}

\end{document}